\documentclass[a4paper,12pt]{article}
\usepackage{amsmath,amsthm,amsfonts,amssymb,bbm}
\usepackage{graphicx,psfrag,color,cite,subcaption}

\numberwithin{equation}{section}

\usepackage[margin=1.2in]{geometry}

\newcommand{\U}{\mathbf{U}}
\newcommand{\D}{\mathbf{D}}

\newtheorem{prop}{Proposition}[section]
\newtheorem{thm}[prop]{Theorem}
\newtheorem{lem}[prop]{Lemma}
\newtheorem{defin}[prop]{Definition}
\newtheorem{cor}[prop]{Corollary}

\newtheorem{algo}[prop]{Algorithm}
\newtheorem{cla}[prop]{Claim}

\newtheorem{rem}[prop]{Remark}

\title{The Preisach graph and longest increasing subsequences}
\author{Patrik L.\ Ferrari\thanks{Institute for Applied Mathematics, Bonn University, Endenicher Allee 60, 53115 Bonn, Germany. E-mail: {\tt ferrari@uni-bonn.de}}\and Muhittin Mungan\thanks{Institute for Applied Mathematics, Bonn University, Endenicher Allee 60, 53115 Bonn, Germany. E-mail: {\tt mungan@iam.uni-bonn.de}}\and
M. Mert Terzi\thanks{LPTMS, CNRS, Universit\'e Paris-Saclay, 91405 Orsay, France. E-mail: {\tt mert.terzi@u-psud.fr}}
}

\date{}

\begin{document}
\maketitle

\sloppy
\begin{abstract}
The Preisach graph is a directed graph associated with a permutation $\rho\in{\cal S}_N$. We give an explicit bijection between its vertices and increasing subsequences of $\rho$ with the property that the length of a subsequence equals to the degree of nesting of the corresponding vertex inside a hierarchy of cycles and sub-cycles of the graph. As a consequence, the nesting degree of the Preisach graph equals the length of the longest increasing subsequence.
\end{abstract}

\section{Introduction and results}

\paragraph{Increasing subsequences in random permutations.}
We consider a permutation $\rho=(\rho_1,\ldots,\rho_N)\in{\cal S}_N$ of $\{1,\ldots,N\}$. The well-known Robinson-Schensted-Knuth correspondence~\cite{Sch61,Knu70} gives a bijection between $\rho$ and a \emph{pair of standard Young tableaux} $(P,Q)$, where the length of the first row of the tableaux equals the length  $\ell(\rho)$ of the longest increasing subsequence of $\rho$. The meaning of the sum of the length of the first $k$th rows in terms of increasing subsequences was unraveled in~\cite{Gre74}. Furthermore, if $\rho$ is taken uniformly distributed, the limiting law of $\ell(\rho)$ has been studied. The law of large numbers was determined in~\cite{LS77,VK77}, see also~\cite{AD95,Sep98,Jo98}.

Relations with random matrix theory have also been established. In~\cite{BDJ99} it was proven that the fluctuations of $\ell(\rho)$ are governed by the GUE Tracy-Widom distribution function~\cite{TW94}, see also the reviews~\cite{AD99,Cor18}. In the proof of the fluctuations, it is convenient to consider $N$ to be Poisson distributed. Then the shape of the Young tableaux is described by the poissonized Plancherel measure, which is a Schur measure arising naturally in other models, such as the Hammersley process~\cite{AD95} or a stochastic growth model of an interface~\cite{PS00,Jo01,BO04}.

In this note, we consider a completely different representation of a permutation: we represent $\rho$ as a directed graph, called the \emph{Preisach graph}, see~\cite{Ter20}. Its vertex set is a subset of $\{-1,1\}^N$, which can therefore be thought of as configurations of $N$ spins. Starting from the $(-1,\ldots,-1)$ spin configuration, each vertex of the graph can be obtained  by sequences of spin flips that follow the rules explained below. In particular, for any given vertex of the graph there is a minimal number of times one switches from $-1\to +1$ transitions to $+1\to -1$ to reach it.

We then analyze the structure of the graph and establish the correspondence to increasing subsequences. Our main result is the derivation of an \emph{explicit bijection} between the set of all increasing subsequences of $\rho$ and the vertex set of the Preisach graph, see Theorem~\ref{thmMain}. The bijection has an interesting geometric property: the minimal number of switches for a vertex is exactly equal to the length of the associated increasing subsequence. As a consequence, the length of the longest increasing subsequence is reflected by the nesting degree of the graph, see Corollary~\ref{CorMain}. Thus the geometric structure of the Preisach graph with its cycles nested inside cycles is directly related to the structure of the permutation $\rho$.

\begin{figure}
\begin{center}
\includegraphics[height=4cm]{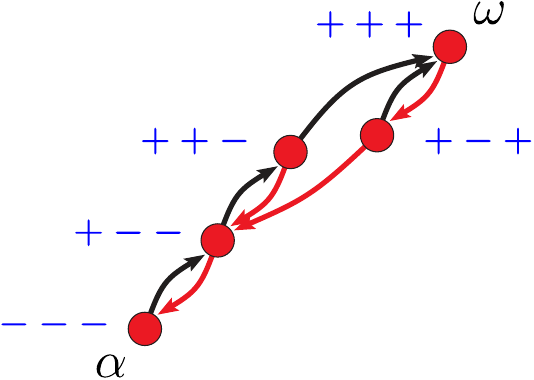}
\caption{The Preisach graph for the permutation $\rho=(2,3,1)$. Transitions under $\U$ and $\D$ are marked by black, respectively red edges. In this and further figures, we never indicate explicitly the transitions $\alpha=\D\alpha$ and $\omega=\U\omega$.}
\label{FigExample}
\end{center}
\end{figure}

\paragraph{The model.}
We start with the description of the Preisach graph $G=(V,E)$, which is a directed graph generated by a permutation $\rho\in{\cal S}_N$.
Figure~\ref{FigExample} depicts the Preisach graph for $N = 3$ and generated by the permutation $\rho = (2,3,1)$.
The vertex set $V$ consists of spin configurations $\sigma=(\sigma_1,\ldots,\sigma_N)\in \Omega = \{-1,1\}^N$. We denote by $\alpha=(-1,\ldots,-1)$ and $\omega=(1,\ldots,1)$ two spin configurations, which will belong to the graph $G$
for any permutation $\rho$.  For the construction of the edge set $E$ of the graph, we need to first define two maps $\textbf{U}$ and $\textbf{D}$ on $\Omega$.
Let $\rho=(\rho_1,\ldots,\rho_N)\in{\cal S}_N$ be a permutation of $\{1,\ldots,N\}$. Given a spin configuration $\sigma$, set
\begin{equation}
\begin{aligned}
i^+(\sigma)&=\min\{i: \sigma_i=-1\},\quad \textrm{for }\sigma\neq\omega,\\
i^-(\sigma)&=\rho_r,\quad r=\min\{s: \sigma_{\rho_s}=1,1\leq s\leq N\},\quad \textrm{for }\sigma\neq\alpha,
\label{eqn:ipmdef}
\end{aligned}
\end{equation}
and denote by $\sigma^i$ the configuration obtained from $\sigma$ by flipping the $i$th spin. We then define
\begin{equation}
\begin{aligned}
\U \sigma &= \sigma^{i^+(\sigma)}\textrm{ if }\sigma\neq\omega, \quad &\U \omega&=\omega,\\
\D \sigma&=\sigma^{i^-(\sigma)}\textrm{ if }\sigma\neq\alpha, \quad &\D \alpha&=\alpha.
\label{eqn:UDdef}
\end{aligned}
\end{equation}
\begin{defin}[Preisach Graph]
With these definitions, $\rho$ defines a directed graph $G=(V,E)$, the Preisach Graph, as follows. The vertex set $V$ consists of all elements $\sigma\in\Omega$ which, starting from $\alpha$, can be reached by a sequence of maps $\U$ and $\D$. Given $\sigma,\tilde\sigma\in V$, there is a directed edge from $\tilde\sigma$ to $\sigma$ if either $\sigma=\U \tilde\sigma$ or $\sigma=\D \tilde\sigma$. We also decompose the set $E=E_U\cup E_D$, where $E_U$ (resp.\ $E_D$) contains the edges generated by a $\U$ (resp.\ $\D$) transition.
\end{defin}

Note that by construction $\omega =\U^N \alpha$,  $\alpha = \D^N \omega $ and, more generally, for any configuration $\theta \in \Omega$, $\U^n\theta=\omega$ and $\D^n\theta=\alpha$ for $n$ large enough.
Observe that, in general, the vertex set $V$ of the Preisach Graph does not contain all elements of $\Omega$, see Figure~\ref{FigExample} for an illustration.

\begin{defin}[Nesting degree]
Define the degree of nesting ${\cal N}(\sigma)$ of $\sigma$ as the minimal number of alternating sequences of $\U $ and $\D $, that is,
\begin{equation}
{\cal N}(\sigma)=\min\{m: \sigma=(\D\textrm{ or }\U)^{n_m}\cdots\D^{n_4}\U^{n_3}\D^{n_2}\U^{n_1}\alpha\}.
\end{equation}
For the graph $G=(V,E)$, we define the maximal degree of nesting by
\begin{equation}
{\cal N}(G)=\max_{\sigma\in V}{\cal N}(\sigma).
\end{equation}
\end{defin}

\paragraph{The results.}
For a permutation $\rho=(\rho_1,\ldots,\rho_N)$, we say that a subsequence $\tilde\rho=(\rho_{i_1},\ldots,\rho_{i_n})$ is increasing, if $i_1<i_2<\ldots<i_n$ and $\rho_{i_1}<\rho_{i_2}<\ldots<\rho_{i_n}$. It has length $\ell(\tilde\rho)=n$. For convenience, we also introduce the trivial subsequence containing
no elements and denote it by $\emptyset$, so that $\ell(\emptyset) = 0$.
 It has been shown in~\cite{Ter20} that the number of increasing subsequences equals the number of vertices in the graph $G$.
\begin{thm}\label{thmMain}
Denote by $\cal I_\rho$ the set of increasing subsequences of $\rho$. There exists an \emph{explicit} bijection $\Phi:V\to {\cal I}_\rho$ with the property
\begin{equation}
\ell(\Phi(\sigma))={\cal N}(\sigma).
\end{equation}
The bijection is given in \eqref{eqBijection}.
\end{thm}

As a straightforward consequence we have the following corollary.
\begin{cor}\label{CorMain}
Given a permutation $\rho\in{\cal S}_N$, the maximal degree of nesting of $G=G(\rho)$ equals the length of the longest increasing subsequence of $\rho$.
\end{cor}

\paragraph{Origin of the Preisach model and related models.}
 The set of states $\Omega$ along with the pair of maps $\U,\D: \Omega \to \Omega$ arises naturally in the description of the dynamics of athermal systems driven by a scalar
 parameter, such as a force applied uni-directionally, and where the response is assumed to be independent of the rate at which the driving parameter is changing in time~\cite{MW19, MT18,MSDR19}.

 In the physics literature this type of dynamical regime is called athermal quasi-static (AQS)~\cite{ML06} and it has been used to numerically model yielding phenomena in diverse materials such as crystals~\cite{ST11,NI19}
 and  amorphous solids~\cite{BN17,RV15,LPS17}.
 Here $\Omega$ denotes the set of quasi-static states, while the
 maps $\U$ and $\D$ describe transitions between these states when the driving parameter is increased, or respectively decreased, just sufficiently enough to trigger such a transition. Moreover, these
 maps are assumed to be acyclic. This implies that the driving is such that, under monotonous increase (decrease) of the driving parameter, a previously visited state cannot be revisited (except for a subset of
 $\D$- and $\U$-absorbing states $\alpha_i = \D \alpha_i$ and $\omega_j = \U \omega_j$). Such systems have been called {\em AQS-automata} and it has been shown that their dynamics can be described in terms of a corresponding graph, the AQS graph~\cite{MT18}.

The Preisach Graph constructed above is an example of such an AQS graph and furnishes  a representation of the dynamics of the
 Preisach model~\cite{Pre35, May86b},
 as shown in~\cite{Ter20}, see also~\cite{Ra16}. The Preisach model has been used to describe a broad range of systems exhibiting hysteresis, including magnetic materials, where the model originated~\cite{Pre35, BEHS83},
 but also fracture in dilatant rocks~\cite{Hol81},
 and more recently, memory formation in matter~\cite{KNPSZ19, HKKW20}.
 A comprehensive review of the Preisach model and its applications can be found in~\cite{BM06,BS12}.
At the same time, the Preisach model is the simplest system exhibiting return-point-memory (RPM)~\cite{BEHS83, DKKRSS93},
a property wherein a system remembers the states at which the direction of an external driving was reversed. As shown in~\cite{MT18}, the presence of RPM imposes strong constraints on the structure of the associated AQS graph. The corresponding graph property has been called the {\em RPM loop-property}  ($\ell$RPM) and will be defined in the next section.

\paragraph{Acknowledgments.} The work of P.L. Ferrari and M. Mungan was partly funded by the Deutsche Forschungsgemeinschaft (DFG, German Research Foundation) under Germany’s Excellence Strategy - GZ 2047/1, projekt-id 390685813 and by the Deutsche Forschungsgemeinschaft (DFG, German Research Foundation) - Projektnummer 211504053 - SFB 1060.
M. Mungan is also funded by the Deutsche Forschungsgemeinschaft (DFG, German Research Foundation) - Projektnummer 398962893.

\section{The Preisach graph}
In this section we derive the relevant properties of the Preisach graph, including $\ell$RPM. We will then show in Section \ref{sectAlgorithm} how these features allow us to construct the Preisach Graph in an iterative manner.

\subsection{The $\ell$RPM property}
One key property used in establishing our main  result is the \emph{loop return-point memory} property ($\ell$RPM)
~\cite{MT18} which the Preisach Graph possesses, and which we define next. Note that the definitions to follow apply to any AQS graph\footnote{Note the slight change in terminology with respect to
 ~\cite{MT18}. In~\cite{MT18} a {\em loop} also refered to what we here call a {\em cycle}. For reasons of clarity, we have chosen to distinguish these here. }.

\newpage
\begin{defin}[Loops and return point memory]\label{def:general}
$ $
\begin{itemize}
\item [(a)] Consider an ordered pair of states $(\mu,\nu)\in \Omega\times \Omega$ satisfying $\nu=\U^n \mu$ for some $n$ and $\mu=\D^m \nu$ for some $m$.
We call the set of vertices of the form $\U^k\mu$, $0\leq k\leq n$ and $\D^j\nu$, $0\leq j\leq m$ the $\U$-, respectively, $\D$-boundary vertices of $(\mu,\nu)$. We then
say that $(\mu,\nu)$ forms a \emph{$\U\D$-cycle}, or simply \emph{cycle}, in the sense that the union of the $\U$- and $\D$-boundaries forms a cycle in the graph and refer to $\mu$ and $\nu$ as the lower and upper endpoint of this cycle.
\item [(b)] For a vertex $\mu$, we call \emph{$\U$-orbit of $\mu$} the sequence of points $\mu,\U\mu,\U^2\mu,\ldots$ and \emph{$\D$-orbit of $\mu$} the sequence of points $\mu,\D\mu,\D^2\mu,\ldots$
\item [(c)] Given a cycle $(\mu,\nu)$, if for each state $u$ on the $\U$-boundary its $\D$-orbit contains $\mu$ and likewise, for each state $v$ on the $\D$-boundary its $\U$-orbit contains $\nu$, then we say that the cycle has the \emph{absorption property}.
This implies in particular that for these states, $(\mu,u)$  and  $(v,\nu)$ each form cycles. Assuming that the cycle $(\mu,\nu)$ possesses the absorption property, we refer to the cycles $(\mu,u)$, $u\in \U$-boundary and $(v,\nu)$, $v\in\D$-boundary as its major sub-cycles.
\item [(d)] Given a cycle $(\mu,\nu)$, we say that it possesses the \emph{loop return-point memory} property  ($\ell$RPM), if it is absorbing and every major sub-cycle has the $\ell$RPM property.
\item [(e)] Consider a graph such that each cycle satisfies the absorption property and let $(\mu,\nu)$ be a cycle. We can associate with the cycle $(\mu,\nu)$ a subgraph, called the \emph{loop $(\mu,\nu)$}, whose
vertex set is given by the following iterative union of sets:
the set of all boundary states of major sub-cycles of $(\mu,\nu)$, the boundary states of the major sub-cycles of each of the major sub-cycles of $(\mu,\nu)$, and so on. The edge set of the loop $(\mu,\nu)$ is inherited from the graph.
\end{itemize}

\end{defin}

We start by establishing the $\ell$RPM property for the cycle  $(\alpha,\omega)$.
\begin{lem}\label{preisachabs}
 Let $\Omega = \{ -1,1\}^N$, $\rho \in \mathcal{S}_N$ and let the pair of maps $\U,\D: \Omega \to \Omega$ be defined as in \eqref{eqn:ipmdef} and \eqref{eqn:UDdef}. Then, given any cycle  $(\mu,\nu) \in \Omega \times \Omega$, $(\mu,\nu)$ has the absorption property, as defined in Definition~\ref{def:general}(c). It follows that $(\mu,\nu)$ has the $\ell$RPM property.
\end{lem}
It follows in particular that the cycle $(\alpha,\omega)$ has the $\ell$RPM property.
\begin{proof}
 Since $(\alpha,\omega)$ is a cycle, we know that given any $\rho$ there exists at least one pair $(\mu,\nu)$ forming a cycle. Let $(\mu,\nu)$ be a cycle. From the definitions of the maps \eqref{eqn:ipmdef} and
 \eqref{eqn:UDdef}, it follows that $\nu = \U^n \mu$ and $\mu = \D^n \nu$ for some $n$. For $p = 1, \ldots, n$, denote by $i_p = i^+(\U^p\mu)$ the indices of the spins that change their state from $-1 \to 1$ as we move along the $\U$-boundary of the cycle. Since $(\mu,\nu)$ forms a cycle, the same set of spins must revert their states from $1 \to -1$ as we move from $\nu$ to $\mu$ along its $\D$-boundary.
 Denote therefore by
 \begin{equation}
 r_1 < r_2 < \ldots < r_n
 \label{indexorder}
 \end{equation}
 the indices of the permutation $\rho$ so that
\begin{equation}
  \{i_1, \ldots, i_n\} = \{ \rho_{r_1}, \ldots, \rho_{r_n} \}.
\end{equation}
Moreover, by choice of ordering of the indices in \eqref{indexorder}, it follows that
\begin{equation}
 \{i_1, \ldots, i_n\} \subset \{ \rho_1, \rho_2, \ldots, \rho_{r_n} \}.
 \label{eqn:NP1}
\end{equation}
Next, observe that, since the pair $(\mu,\nu)$ forms a cycle, it must be that
\begin{equation}
 \nu_j = -1, \quad \forall j \in \{\rho_1, \rho_2, \ldots, \rho_{r_n} \} \setminus \{ i_1, i_2, \ldots, i_n \},
 \label{eqn:NP2}
\end{equation}
since otherwise some of these sites would flip their spins to $-1$ before the sites $i_1, \ldots, i_n$ and there would be no $\D$-orbit from $\nu$ leading to $\mu$. This in turn implies
\begin{equation}
\mu_{j} = -1, \quad \forall j \in \{ \rho_1, \rho_2, \ldots, \rho_{r_n} \}.
 \label{eqn:NP3}
\end{equation}

Take now  any $\U$-boundary state $\U^k \mu$ and consider its $\D$-orbit. We claim that
\begin{equation}
 \{ i_1, i_2, \ldots, i_k \} = \{ i^-(\U^k \mu), i^-(\D\U^k \mu), \ldots, i^-(\D^{k-1}\U^k \mu) \},
\end{equation}
from which it then follows that $\mu = \D^k \U^k \mu$.

To prove the claim, note first that $i_1, i_2, \ldots, i_k$ are the sites that flip their spin to $+1$ as we move to $\U^k \mu$ along the $\U$-boundary of $(\mu,\nu)$. From \eqref{eqn:NP1}
and
\eqref{eqn:NP3} it follows that $i_1 < i_2 < \ldots < i_k$ are the only sites in $ \{\rho_1, \rho_2, \ldots, \rho_{r_n} \} $ whose spin is now in state $1$. Hence the spins of these sites must be flipped to $-1$ first and the claim follows.

The proof that the $\U$-orbit of any state $v$ on the $\D$-boundary leads to $\nu$ follows a similar line of reasoning and we omit it. Alternatively, it follows from the above, using the $\U$-$\D$ symmetry mentioned in Remark \ref{rem:UDsym}.
\end{proof}

\subsection{Properties of the Preisach graph}
First we want to see how the Preisach graph for $N$ spins is related to the Preisach graph with $N-1$ spin. For this purpose, consider a permutation $\rho\in{\cal S}_N$.

\begin{lem}\label{lem1}
Let us denote by $k$ the number such that $\rho_k=N$. Then we have the property
\begin{equation}
\D^{k-1}\U^{N-1}\alpha = \D^k \U^N\alpha = \D^k \omega,
\end{equation}
see Figure~\ref{Fig1}(a) for an illustration.
\end{lem}
\begin{figure}
\centering
\begin{subfigure}[b]{.4\textwidth}
\centering
\includegraphics[height=3.5cm]{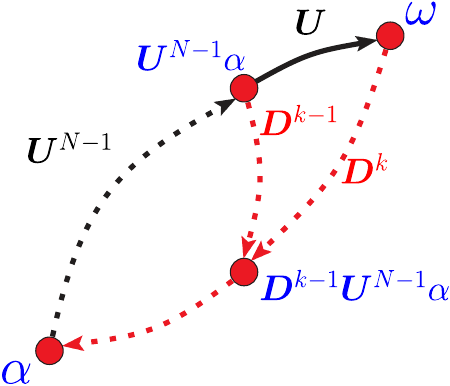}
\caption{}
\end{subfigure}~
\begin{subfigure}[b]{.4\textwidth}
\centering
\includegraphics[height=4.5cm]{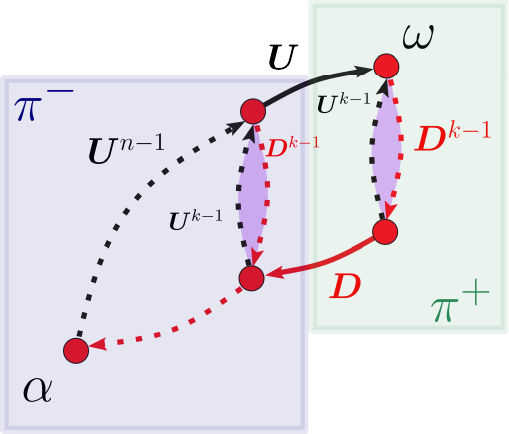}
\caption{}
\end{subfigure}
\caption{(a) Graphical illustration of Lemma~\ref{lem1}. (b) Decomposition used in the forward construction of the Preisach graph.}
\label{Fig1}
\end{figure}
\begin{proof}
The state $\U^{N-1}\alpha$ has spin configuration $(1,\ldots,1,-1)$, and applying one more time $\U$, we get $\omega$, i.e., $\U^N\alpha=\omega$. The state $\D^{k-1} \U^{N-1}\alpha$ is obtained from $\U^{N-1}\alpha$ by flipping all the $+1$ spins at sites $j<\rho_k$ back to $-1$, while the state $\D^k\U^N\omega$ is obtained from $\omega$ by flipping all the $+1$ spins at sites $j\leq \rho_k=N$ back to $-1$. Thus these two states are the same.
\end{proof}
\begin{rem}\label{rem:UDsym}
The Preisach graph has an obvious $\U-\D$ symmetry. Indeed, the Preisach graph obtained by replacing every $\U$-transition by a $\D$-transition and vice versa is the Preisach graph generated by the inverse permutation.
\end{rem}
Using this fact one readily obtains the following result.
\begin{lem}\label{lem2}
Let  $\rho_N = q$. Then we have the property
\begin{equation}
\U^{q-1} \D^{N-1}\omega=\U^q \D^N \omega=\U^q \alpha.
\end{equation}
\end{lem}

By the $\ell$RPM property and with $k$ such that $\rho_k = N$, each of the pairs
\begin{equation}
(\alpha,\U^{N-1}\alpha),\quad (\D^{k-1}U^{N-1}\alpha,\U^{N-1}\alpha),\quad (\D^{k-1}\omega,\omega)
\end{equation}
form a loop with the following properties.
\begin{lem}\label{Lemmaproperties}
Consider the Preisach graph generated by the permutation $\rho=(\rho_1,\ldots,\rho_N)$, and let $k$ be such that $\rho_k = N$. Then the following hold.\\
(a) The loop $(\alpha,\U^{N-1}\alpha)$ projected onto the first $N-1$ spins\footnote{With this we mean that every elements $(\sigma_1,\ldots,\sigma_{N-1},\sigma_N)$ is projected onto $(\sigma_1,\ldots,\sigma_{N-1})$.}  is the Preisach graph of $N-1$ spins for the permutation $\rho_-$ obtained from $\rho$ by removing $\rho_k=N$:
\begin{equation}
\rho_-=(\rho_1,\ldots,\rho_{k-1},\rho_{k+1},\ldots,\rho_N).
\end{equation}
(b) The loops $(\D^{k-1}\U^{N-1}\alpha,\U^{N-1}\alpha)$ and $(\D^{k-1}\omega,\omega)$ are isomophic and disjoint. Furthermore, they are in bijection with a system of $k-1$
spins and permutation $\rho_+$ obtained by keeping the first $k-1$ elements of $\rho$:
\begin{equation}
\rho_+=(\rho_1,\ldots,\rho_{k-1}).
\end{equation}
(c) The loops $(\alpha, \U^{N-1}\alpha)$ and $(\D^{k-1}\omega,\omega)$ are disjoint. The full graph is obtained by the union of these two loops joined by two directed edges: the one from $\U^{N-1}\alpha$ to $\U^N\alpha=\omega$ and the one from $\D^{k-1}\omega$ to $\D^k\omega=\D^{k-1}\U^{N-1}\alpha$.\\
(d) The Preisach graph is planar.
\end{lem}
\begin{proof}
(a) This follows by noticing that the loop $(\alpha,\U^{N-1}\alpha)$ contains all elements obtained from $\alpha$ by $\U$- and $\D$-transitions which keep the $N$th spin equal to $-1$.

(b) Elements of the two loops are disjoint, since every element of $(\D^{k-1}U^{N-1}\alpha,\U^{N-1}\alpha)$ has the $N$th spin equal to $-1$, while every element of $(\D^{k-1}\omega,\omega)$ has the $N$th spin equal to $+1$. The bijection between the two loops is simply obtained by the flip of spin $N$.
Furthermore, for all elements $\sigma$ of $(\D^{k-1}\omega,\omega)$ we have that $\sigma_{\rho_k} = \sigma_{\rho_{k+1}} = \ldots = \sigma_{\rho_N} = 1$, where again $k$ is such that $\rho_k = N$. Thus the projection of $(\D^{k-1}\omega,\omega)$ to the spins $\sigma_{\rho_1},\ldots,\sigma_{\rho_{k-1}}$ equals to the Preisach graph of these $k-1$ spins and the associated permutation is $\rho_+$.

(c) The two loops are disjoint, since in the first $\sigma_N=-1$, while in the second one $\sigma_N=+1$. Furthermore, in order to go from the loop $(\alpha, \U^{N-1}\alpha)$ to the loop $(\D^{k-1}\omega,\omega)$, the $N$th spin has to be switched from $-1$ to $1$. Due to the definition of the $\U$-transition, this can happen only once all other spins have value $+1$, i.e., only from the state $\U^{N-1}\alpha=(1,\ldots,1,-1)$. Similarly, the transition for the $N$th spin from $1$ to $-1$ can occur only from the state $\D^{k-1}\omega$.

(d) This follows by iterating property (c).
\end{proof}

Property (c) allows us to decompose the Preisach graph of $N$ spins into two Preisach graphs $\pi_-$ and $\pi_+$, with the first one being in bijection with the system of $N-1$ spins generated by $\rho_-$, while the second graph is in bijection with the system of $k-1$ spins generated by $\rho_+$. The decomposition is illustrated in Figure~\ref{Fig1}(b).

Furthermore, by the $\ell$RPM property, this decomposition can be iteratively continued until all sup-loops consist of a single state. This leads to the following Lemma.
\begin{lem}
The Preisach graph $G=(V,E)$ coincides with the loop $(\alpha,\omega)$.
\end{lem}

\begin{proof}
 From Lemma \ref{Lemmaproperties} it follows immediately that the loop $(\alpha, \omega)$ must be a subgraph of $G$. For the converse statement, consider any $\sigma \in V$ so
 that
 \begin{equation}
  \sigma = (\D\textrm{ or }\U)^{n_m}\cdots\D^{n_4}\U^{n_3}\D^{n_2}\U^{n_1}\alpha
 \end{equation}
 for some $n_1, n_2, \ldots n_m$. The sequence of transitions given above describes a path from $\alpha$ to $\sigma$. Using the $\ell$RPM property of $(\alpha, \omega)$, and proceeding inductively, one obtains that each of the states visited on this path must be a $\U$- or $\D$-boundary state of some sub-cycle of $(\alpha, \omega)$.
\end{proof}

\subsection{Construction of the Preisach graph}\label{sectAlgorithm}
The properties obtained above permit an iterative construction of the Preisach graph associated with the permutation $\rho=(\rho_1,\ldots,\rho_N)$. Let us define a sequence of permutations $\rho^{(n)}$ of $\{1,\ldots,n\}$, $n=1,\ldots,N$ by selecting from $\rho$ the entries belonging to $\{1,\ldots,n\}$ only. By Lemma~\ref{Lemmaproperties} (c), the graph for $\rho^{(n+1)}$ is obtained from its loop $\pi_-^{(n+1)}$ and $\pi_+^{(n+1)}$ together with two directed edges that join them. Moreover, the graph of $\pi_-^{(n+1)}$ is isomorphic to the graph of $\rho^{(n)}$, while  $\pi_+^{(n+1)}$ is isomorphic to a certain sub-loop of  $\rho^{(n)}$ as explained in the algorithm below.

Thus, the graph at step $n+1$ associated with $\rho^{(n+1)}$ can be constructed by the following algorithm.

\newpage
\begin{algo}[Forward algorithm]
The forward algorithm to construct the Preisach graph of a permutation $\rho$ is the following:
\begin{enumerate}
\item Let $k$ be defined by $\rho^{(n+1)}_{k}=n+1$.
\item Start with the graph for $\rho^{(n)}$, which is the loop $(\alpha,\U^{n}\alpha)$.
\item Create a copy of the loop $(\D^{k-1}\U^{n}\alpha,\U^n\alpha)$ and call its end-points $\D^{k-1}\U^{n+1}\alpha$ and $\U^{n+1}\alpha$ respectively.
\item Add the directed edge from $\U^n\alpha$ to $\U^{n+1}\alpha$, and from $\D^{k-1}\U^{n+1}\alpha$ to $\D^{k-1}\U^n\alpha$.
\end{enumerate}
\end{algo}

\begin{rem}\label{rem:UDsymB}
Using the $\U-\D$ symmetry mentioned already in Remark~\ref{rem:UDsym}, an alternative algorithm to construct the Preisach graph consists in considering the sequence of permutations $\tilde\rho^{(n)}=(\rho_1,\ldots,\rho_n)$, for $n=1,\ldots,N$. At step $n$ we attach a
copy of a subloop of the current graph at its lower endpoint to obtain the graph at step $n+1$. See Figure~\ref{Fig3} for an illustration for the permutation $\rho=(2,4,3,1)$.
\end{rem}

\begin{figure}[h!]
\begin{center}
\includegraphics[width=\textwidth]{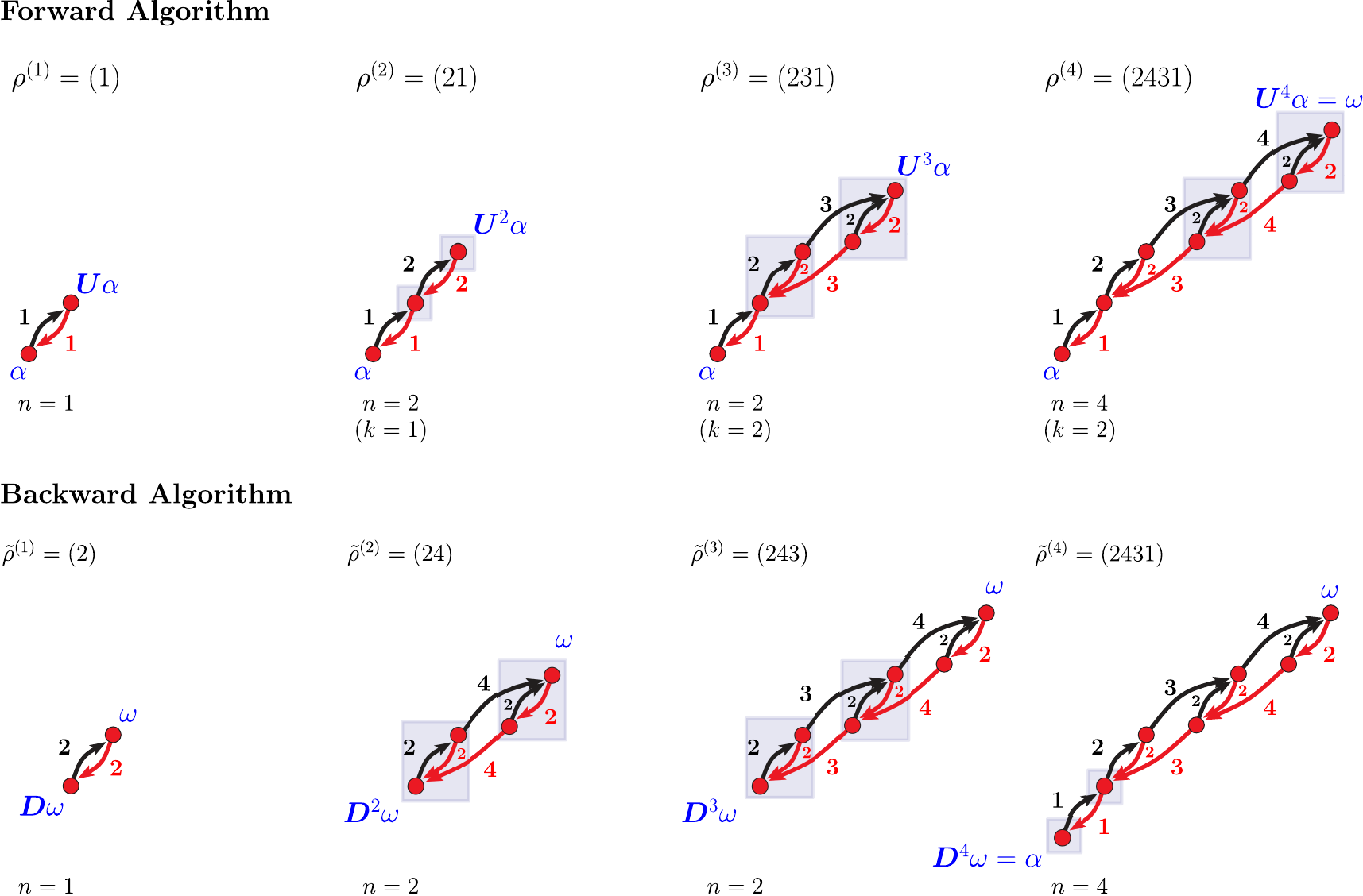}
\caption{Illustration of forward and backward algorithms for $\rho=(2,4,3,1)$. The copied loops are highlighted in gray and retain the labeling of their edges. }
\label{Fig3}
\end{center}
\end{figure}

\newpage
\section{Proof of the results}

\paragraph{Labeling of the transitions.}
To define the bijection between vertices of the Preisach graph and increasing subsequences, we first add labels to each edge of the graph. The label indicates which spin is changing its sign during the
corresponding transition, as determined by $i^\pm$ in \eqref{eqn:ipmdef} and \eqref{eqn:UDdef}. Alternatively, this labeling can be carried out during the algorithmic construction of the Preisach graph described in Section~\ref{sectAlgorithm}: the directed edges between the loops $\pi_-^{(n+1)}$ and $\pi_+^{(n+1)}$ get the label $n+1$, while the labels of $\pi_\pm^{(n+1)}$ are simply inherited by the ones of the graph at step $n$. That is, the edges which were already present are unchanged and the new loop $\pi_+^{(n+1)}$, which is a copy of loop of the graph at step $n$, inherits the edge labels (see the top row of Figure~\ref{Fig3} for an example).

\paragraph{Shortest path to a vertex.}
\begin{prop}
For any vertex $\sigma\in V$, there exists a \emph{unique} shortest path from $\alpha$ to $\sigma$.
\end{prop}
\begin{proof}
First of all, consider only paths from $\alpha$ to $\sigma$ which do not use twice the same edge, since otherwise they would not be shortest. Assume that the shortest path is not unique and let $\gamma_1$ and $\gamma_2$ be two such paths.

Let us show first that the shortest paths cannot split on the $\U$-orbit of $\alpha$, i.e., that they must coincide on this orbit. Assume that they split at the vertex $\U^{n-1}\alpha$, with the path $\gamma_1$ going through the edge from $\U^{n-1} \alpha$ to $\D\U^{n-1}\alpha$, while the path $\gamma_2$ passes through the edge from $\U^{n-1}\alpha$ to $\U^n\alpha$. Let $k$ be the index such that $\rho_k=n$. By Lemma~\ref{lem1} and Lemma~\ref{Lemmaproperties}(a), it follows that $\D^{k-1}\U^{n-1}\alpha=\D^{k}\U^{n}\alpha$. Then, the end-point $\sigma$ cannot belong to the loop $(\D^{k-1}\U^n\alpha,\U^n\alpha)$, since these vertices cannot be reached by the path $\gamma_1$. Indeed, for such a path to reach this loop it needs to pass a second time by $\U^{n-1}\alpha$, which would imply that it is not a shortest path. Also, $\sigma$ cannot be in the loop $(\D^{k-1}\U^{n-1}\alpha,\U^{n-1}\alpha)$, since the path $\gamma_2$ to reach one of these points needs to pass by the vertex $\D^{k-1}\U^{n-1}\alpha$. But the loops $(\D^{k-1}\U^{n-1}\alpha,\U^{n-1}\alpha)$ and $(\D^{k-1}\U^n\alpha,\U^n\alpha)$ being isomorphic, for each path $\gamma_2$ there would exists a path $\tilde{\gamma}_2$ of length decreased by two, which contradicts the assumption that $\gamma_2$ is a shortest path.

Consequently, there is a unique point on the $\U$-orbit of $\alpha$ where the shortest path to $\sigma$ switches from a $\U$ to a $\D$ orbit, which we denote by $s_1=\U^{n_1}\alpha$. We call such states {\em switch-back} states. Also note that the path $\gamma_1$ passes exactly once by an edge with label $n_1$ and it never uses edges of higher label, as we will prove below. Next one repeats the same argument with the initial point being $s_1$ and the roles of $\U$ and $\D$ interchanged.
By the $\U$-$\D$ symmetry, this leads to a unique point on the $\D$-orbit of $s_1$ where it switches from a $\D$ to a $\U$ orbit. We denote this vertex by $s_2=\D^{n_2} s_1$ and clearly $n_2<n_1$. Repeating the first argument with $\alpha$ replaced by $s_2$ we get a vertex $s_3=\U^{n_3} s_2$ with $n_3<n_2$, and so on, until we reach the end-point $\sigma$.
\end{proof}

\paragraph{The explicit bijection.}
Given the uniqueness of the shortest path, this path can be described as follows.
\begin{defin}\label{DefShortest}
Let $\gamma$ be the shortest path from $\alpha$ to $\sigma$. We call switch-back states for $\sigma$ the vertices where the shortest path switches from a $\U$ to a $\D$ transition or vice versa.
We will regard the destination state $\sigma$ as a switch-back state as well. Denote by $\ell_1,\ldots,\ell_m$ the labels of the edges on the path $\gamma$ ending at the switch-back states, ordered according to their appearance on $\gamma$.
\end{defin}
Define the map $\Phi:E\to {\cal I}_\rho$ by
\begin{equation}\label{eqBijection}
\Phi(\sigma)=(\ell_m,\ldots,\ell_2,\ell_1),
\end{equation}
see Figure~\ref{FigShortestPath} for an illustration of the bijection. We are now ready to prove the main result.

\begin{figure}
\begin{center}
\includegraphics[width=\textwidth]{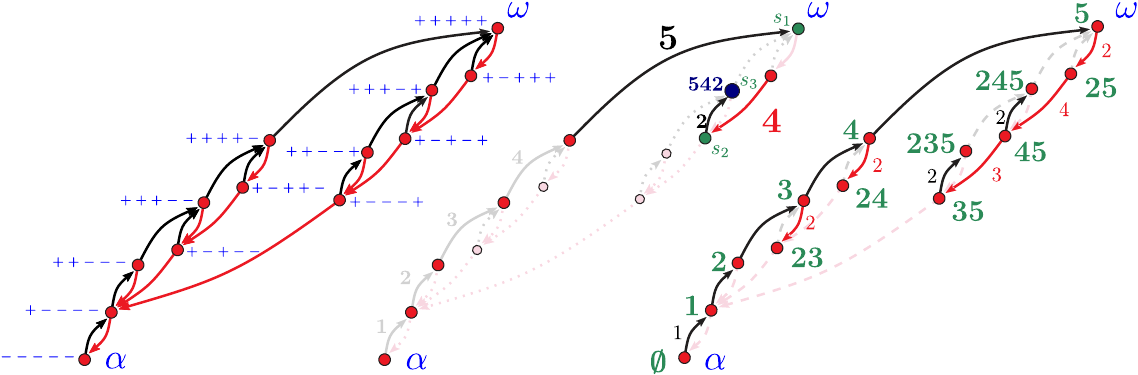}
\caption{From left to right: the Preisach graph generated by the the permutation $\rho=(2,4,3,5,1)$,  the shortest path to the vertex associated with the increasing subsequence $(2,4,5)$  of $\rho$, and the bijection between vertices and increasing subsequences.}
\label{FigShortestPath}
\end{center}
\end{figure}

\begin{proof}[Proof of Theorem~\ref{thmMain}]
Consider a shortest path $\gamma$ from $\alpha$ to $\sigma$, so that there exists some $m$ and numbers $n_1,\ldots,n_m$ such that
\begin{equation}\label{eq3}
\sigma=(\D\textrm{ or }\U)^{n_m}\cdots\D^{n_4}\U^{n_3}\D^{n_2}\U^{n_1}\alpha.
\end{equation}
The condition that $\gamma$ is the shortest path gives $m={\cal N}(\sigma)$. What remains to be shown is that $\Phi$ is a bijection.

\emph{Step 1: First switch-back after $n_1$ $\U$-transitions.} Set $k_1$ such that $\rho_{k_1}=n_1$. By Lemma~\ref{Lemmaproperties}(a) applied repeatedly, $(\alpha,\U^{n_1}\alpha)$ is the graph for the permutation $\tilde\rho$ obtained by keeping only the entries of $\rho$ with $\rho_i\leq n_1$. Thus, every edge in $(\alpha,\U^{n_1}\alpha)$ has a label in $\{1,\ldots,n_1\}$. In particular, by Lemma~\ref{lem1}, we have
\begin{equation}\label{eq2}
\D^{k_1}\U^{n_1}\alpha=\D^{k_1-1}\U^{n_1-1}\alpha.
\end{equation}
Thus we have $\ell_1=n_1$ and every $\sigma$ with $\ell_1=n_1$ will correspond to an increasing subsequence with largest entry $\ell_1$ and vice versa.

\emph{Step 2: Second switch-back after $n_2$ $\D$-transitions.} We have $n_2<k_1$, since otherwise there would be a shorter path leading to $\sigma$, due to \eqref{eq2}. Thus,
\begin{equation}
\sigma\in (\D^{k_1-1}\U^{n_1}\alpha,\U^{n_1}\alpha).
\end{equation}
The labels in this loop are all strictly less than $n_1$: by Lemma~\ref{Lemmaproperties}(b) this loop is isomorphic to the loop $(\D^{k_1-1}\U^{n_1-1}\alpha,\U^{n_1-1}\alpha)$, which has labels in $\{1,\ldots,n_1\}$, and by the construction explained above, the edge labels are also the same. Consequently, the second label is $\ell_2=k_2=\tilde\rho_{n_2}<\ell_1$.

Furthermore, let $k_2=\tilde\rho_{n_2}$. Then,
\begin{equation}
\sigma \in (\D^{n_2}\U^{n_1}\alpha,\U^{k_2-1}\D^{n_2}\U^{n_1}\alpha).
\end{equation}
Indeed, if the path $\gamma$ would go through $\U^{k_2}\D^{n_2}\U^{n_1}\alpha$, then it would not be a shortest one, since by Lemma~\ref{lem2}, $\U^{k_2}\D^{n_2}\U^{n_1}\alpha=\U^{k_2-1}\D^{n_2-1}\U^{n_1}\alpha$.

Every $\sigma$ with $\ell_1=n_1$ and $\ell_2=\tilde\rho_{n_2}$ will correspond to an increasing subsequence with the last two entries $\ell_2,\ell_1$, and vice versa, every increasing subsequence with the last two entries $\ell_2,\ell_1$ corresponds to the choice $n_1=\ell_1$ and $n_2$ such that $\ell_2=\tilde\rho_{n_2}$.

\emph{Further iterations.} Notice that the loop $(\D^{n_2}\U^{n_1}\alpha,\U^{k_2-1}\D^{n_2}\U^{n_1}\alpha)$ is isomorphic to the loop $(\alpha,\U^{k_2-1}\alpha)$. Moreover, by the construction presented above, these loops  have the same labeling of edges (see also the Forward Algorithm in Figure \ref{Fig3} for an illustration). Repeating \emph{Step 1} and \emph{Step 2}, we obtain that for each $\sigma$ there is a different labeling $\ell_1,\ldots,\ell_m$, and by reading from right to left, it is an increasing subsequence of $\rho$. Conversely, for each increasing subsequence there is exactly one choice of $n_1,\ldots,n_m$ such that \eqref{eq3} hold.
\end{proof}

\begin{rem}
In~\cite{Ter20} it was already proven that the cardinality of the set of increasing subsequences contained in $\rho$ equals to the number of vertices of the corresponding Preisach graph. This set being finite, it would have been enough to show that to each $\sigma$ there corresponds a different increasing subsequence.
\end{rem}

%\bibliographystyle{patplain}
%\bibliography{Biblio}

\begin{thebibliography}{10}

\bibitem{AD95}
D.J. Aldous and P.~Diaconis, \emph{Hammersley's interacting particle process
  and longest increasing subsequences}, Probab. Theory Relat. Fields
  \textbf{103} (1995), 199--213.

\bibitem{AD99}
D.J. Aldous and P.~Diaconis, \emph{Longest increasing subsequences: from
  patience sorting to the {B}aik-{D}eift-{J}ohansson theorem}, Bull. Amer.
  Math. Soc. \textbf{36} (1999), 413--432.

\bibitem{BDJ99}
J.~Baik, P.A. Deift, and K.~Johansson, \emph{On the distribution of the length
  of the longest increasing subsequence of random permutations}, J. Amer. Math.
  Soc. \textbf{12} (1999), 1119--1178.

\bibitem{BEHS83}
J.A. Barker, D.E. Schreiber, B.G. Huth, and D.H. Everett, \emph{Magnetic
  hysteresis and minor loops: models and experiments},  \textbf{386} (1983),
  251--261.

\bibitem{BM06}
G.~Bertotti and I.D.~Mayergoyz (editors), \emph{The science of hysteresis:
  Hysteresis in materials}, vol.~3, Gulf Professional Publishing, 2006.

\bibitem{BN17}
D.~Bonn, M.M. Denn, T.~Divoux, L.~Berthier, and S.~Manneville, \emph{Yield
  stress materials in soft condensed matter}, Rev. Mod. Phys. \textbf{89}
  (2017), 035005.

\bibitem{BO04}
A.~Borodin and G.~Olshanski, \emph{{Stochastic dynamics related to Plancherel
  measure}}, AMS Transl.: Representation Theory, Dynamical Systems, and
  Asymptotic Combinatorics (V.~Kaimanovich and A.~Lodkin, eds.), 2006,
  pp.~9--22.

\bibitem{BS12}
M.~Brokate and J.~Sprekels, \emph{Hysteresis and phase transitions}, vol.~121 Springer Science \& Business Media,
2012.

\bibitem{Cor18}
I.~Corwin, \emph{{Commentary on ``Longest increasing subsequences: from
  patience sorting to the Baik-Deift-Johansson theorem'' by David Aldous and
  Persi Diaconis}}, Bull. Amer. Math. Soc. \textbf{55} (2018), 363--374.

\bibitem{NI19}
K.A. Dahmen, J.R. Greer, D.B. Liarte, L.W. McFaul, X.~Ni, J.P. Sethna, and
  H.~Zhang, \emph{Yield precursor dislocation avalanches in small crystals: The
  irreversibility transition}, Phys. Rev. Lett. \textbf{123} (2019), 035501.

\bibitem{Gre74}
C.~Greene, \emph{An extension of {S}chensted's theorem}, Adv. Math. \textbf{14}
  (1974), 254--265.

\bibitem{Hol81}
D.J. Holcomb, \emph{Memory, relaxation, and microfracturing in dilatant rock},
  J. Geophys. Res: Solid Earth \textbf{86} (1981), 6235--6248.

\bibitem{Jo98}
K.~Johansson, \emph{The longest increasing subsequence in a random permutation
  and a unitary random matrix model}, Math. Res. Lett. \textbf{5} (1998),
  63--82.

\bibitem{Jo01}
K.~Johansson, \emph{Discrete orthogonal polynomial ensembles and the
  {P}lancherel measure}, Annals of Math. \textbf{153} (2001), 259--296.

\bibitem{HKKW20}
N.C. Keim, J.~Hass, B.~Kroger, and D.~Wieker, \emph{Global memory from local
  hysteresis in an amorphous solid}, Phys. Rev. Research \textbf{2} (2020),
  012004.

\bibitem{KNPSZ19}
N.C. Keim, J.D. Paulsen, Z.~Zeravcic, S.~Sastry, and S.R. Nagel, \emph{Memory
  formation in matter}, Rev. Mod. Phys. \textbf{91} (2019), 035002.

\bibitem{Knu70}
D.E. Knuth, \emph{{Permutations, matrices and generalized Young tableaux}},
  Pacific J. Math. \textbf{34} (1970), 709--727.

\bibitem{LPS17}
P.~Leishangthem, A.D.S. Parmar, and S.~Sastry, \emph{The yielding transition in
  amorphous solids under oscillatory shear deformation}, Nat. Comm. \textbf{8}
  (2017), 1--8.


\bibitem{ML06}
A.~Lema{\^i}tre and C.E. Maloney, \emph{Amorphous systems in athermal,
  quasistatic shear}, Phys. Rev. E \textbf{74} (2006), 016118.

\bibitem{LS77}
 B.~F.~Logan and L.~A.~Shepp,
 \emph{A variational problem for random {Y}oung tableaux},
 Adv. Math. \textbf{26} (1977), 206--222.

\bibitem{May86b}
I.D. Mayergoyz, \emph{Mathematical models of hysteresis}, Phys. Rev. Lett.
  \textbf{56} (1986), 1518--1521.

\bibitem{MSDR19}
M.~Mungan, S.~Sastry, K.~Dahmen, and I.~Regev, \emph{Networks and hierarchies:
  How amorphous materials learn to remember}, Phys. Rev. Lett. \textbf{123}
  (2019), 178002.

\bibitem{MT18}
M.~Mungan and M.M. Terzi, \emph{The structure of state transition graphs in
  hysteresis models with return point memory: {I}. {G}eneral theory}, Ann.
  Henri Poincar\'e \textbf{20} (2019), 2819 -- 2872.

\bibitem{MW19}
M.~Mungan and T.A. Witten, \emph{Cyclic annealing as an iterated random map},
  Phys. Rev. E \textbf{99} (2019), 052132.

\bibitem{PS00}
M.~Pr{\"a}hofer and H.~Spohn, \emph{Universal distributions for growth
  processes in 1+1 dimensions and random matrices}, Phys. Rev. Lett.
  \textbf{84} (2000), 4882--4885.

\bibitem{Pre35}
F.~Preisach, \emph{{\"U}ber die magnetische {N}achwirkung}, Zeitschrift f{\"u}r
  Physik \textbf{94} (1935), 277--302.

\bibitem{Ra16}
D.~Rachinskii, \emph{Realization of arbitrary hysteresis by low-dimensional gradient flow},
Discr. \& Cont. Dyn. Syst. \textbf{B21} (2016) 227 -- 243.

\bibitem{RV15}
I.~Regev, J.~Weber, C.~Reichhardt, K.A. Dahmen, and T.~Lookman,
  \emph{Reversibility and criticality in amorphous solids}, Nat. Comm.
  \textbf{6} (2015), 1--8.

\bibitem{ST11}
O.U. Salman and L.~Truskinovsky, \emph{Minimal integer automaton behind crystal
  plasticity}, Phys. Rev. Lett. \textbf{106} (2011), 175503.

\bibitem{Sch61}
C.~Schensted, \emph{Longest increasing and decreasing subsequences}, Canad. J.
  Math. \textbf{16} (1961), 179--191.

\bibitem{Sep98}
T.~Sepp{\"a}l{\"a}inen, \emph{Hydrodynamic scaling, convex duality and
  asymptotic shapes of growth models}, Markov Process. Related Fields
  \textbf{4} (1998), 1--26.

\bibitem{DKKRSS93}
J.P. Sethna, K.~Dahmen, S.~Kartha, J.A. Krumhansl, B.~W. Roberts, and J.D.
  Shore, \emph{Hysteresis and hierarchies: Dynamics of disorder-driven
  first-order phase transformations}, Phys. Rev. Lett. \textbf{70} (1993),
  3347.

\bibitem{Ter20}
M.M. Terzi and M.~Mungan, \emph{The state transition graph of the {P}reisach
  model and the role of return point memory}, preprint: arXiv:2001.08486
  (2020).

\bibitem{TW94}
C.A. Tracy and H.~Widom, \emph{{Level-spacing distributions and the Airy
  kernel}}, Comm. Math. Phys. \textbf{159} (1994), 151--174.

\bibitem{VK77}
A.M. Vershik and S.V. Kerov, \emph{Asymptotics of {P}lancherel measure of
  symmetric group and the limiting form of {Y}oung tables}, Sov. Math. Dolk.
  \textbf{18} (1977), 527--531.

\end{thebibliography}

\end{document}